\documentclass[11pt,runningheads]{llncs}

\usepackage[ruled, linesnumbered, comments numbered]{algorithm2e}
\usepackage{graphics}
\usepackage{amsmath}
\usepackage{amsfonts}
\usepackage{amssymb}
\usepackage{graphicx}
\usepackage{psfrag}
\usepackage[whileod]{chl-alg}
\usepackage{pst-tree}
\usepackage{lscape}
\usepackage{subfig}
\usepackage{verbatim}

\newcommand{\dd}{\mathinner{\ldotp\ldotp}}
\newcommand{\cpos}{\mathit{P}}      

\title{Quasiperiodicities in Fibonacci strings}
\author{Michalis Christou\inst{1}, Maxime Crochemore \inst{1,2}, Costas S. Iliopoulos\inst{1,3}}

\institute{%
$\!^1\ ${King's College London, London WC2R 2LS, UK}\\
\email {\{michalis.christou,maxime.crochemore,csi\}@dcs.kcl.ac.uk}\\
$\!^2\ ${Universit\'e Paris-Est, France}\\
$\!^3\ ${Digital Ecosystems \& Business Intelligence Institute, Curtin University\\ GPO Box U1987 Perth WA 6845, Australia}\\
}

\begin{document}
\maketitle

\begin{abstract}
We consider the problem of finding quasiperiodicities in a Fibonacci string.
A factor $u$ of a string $y$ is a \emph{cover} of $y$ if every letter of $y$ falls within some occurrence of $u$ in $y$.
A string $v$ is a \emph{seed} of $y$, if it is a cover of a superstring of $y$.
A left seed of a string $y$ is a prefix of $y$ that it is a cover of a superstring of $y$.
Similarly a right seed of a string $y$ is a suffix of $y$ that it is a cover of a superstring of $y$.
In this paper, we present some interesting results regarding quasiperiodicities in Fibonacci strings, 
we identify all covers, left/right seeds and seeds of a Fibonacci string and all covers of a circular Fibonacci string.
\end{abstract}

\section*{Introduction}
The notion of periodicity in strings is well studied in many fields like
combinatorics on words, pattern matching, data compression and automata
theory (see \cite{Lot01,Lot05}), because it is of paramount importance
in several applications, not to talk about its theoretical aspects.

The concept of quasiperiodicity is a generalization of the notion
of periodicity, and was defined by Apostolico and Ehrenfeucht in \cite{162414}.
In a periodic repetition the occurrences of the single periods do not overlap.
In contrast, the quasiperiods of a quasiperiodic string may overlap.
We call a factor $u$ of a nonempty string $y$ a cover of $y$, if every letter
 of $y$ is within some occurrence of $u$ in $y$. Note that we consider the {\em aligned covers},
where the cover $u$ of $y$ needs to be a border (i.e. a prefix and a suffix) of $y$. 
{\em Seeds} are regularities of strings strongly related to the notion of
cover, as it a seed is a cover of a superstring of the word. They were first defined and studied by Iliopoulos, Moore and Park \cite{Iliopoulos96}.
A left seed of a string $y$, firstly defined in \cite{christou2011efficient}, is a prefix of $y$ that is a cover of a superstring of $y$. Similarly a right seed of a string $y$, also firstly defined in \cite{christou2011efficient}, is a suffix of $y$ that is a cover of a superstring of $y$.

A fundamental problem is to find all covers of a string. A linear time algorithm was given by Moore and Smyth\cite{moore1994optimal}, Li and Smyth\cite{cover_bill} (this algorithm gives also all the covers for every prefix of the string) and an $O(log(log(|y|)))$ work optimal parallel algorithm was given later by Iliopoulos and Park\cite{iliopoulos1996work}. The corresponding problem on seeds is harder, at the moment the fastest and only algorithm is by Iliopoulos, Moore and Park\cite{Iliopoulos96}, running in $O(|y|log|y|)$. 

Fibonacci strings are important in many concepts\cite{berstelfibonacci} and are often cited as a worst case example for many string algorithms. Over the years much scientific work has been done on them, e.g. locating all factors of a Fibonacci string\cite{chuan2005locating}, characterizing all squares of a Fibonacci string\cite{fraenkel1999exact,iliopoulos1997characterization}, identifying all covers of a circular Fibonacci string\cite{iliopoulos1998covers}, identifying all borders of a Fibonacci string\cite{cummings1996borders}, finding palindromes of a Fibonacci word\cite{droubay1995palindromes}, etc.

In our paper we identify all left/right seeds, covers and seeds of a Fibonacci string as well as all covers of a circular Fibonacci string, using a different approach than that of Iliopoulos, Moore and Smyth\cite{iliopoulos1998covers}. It is important to note that we restrict to those quasiperiodicities that are substrings of the Fibonacci string.

The rest of the paper is structured as follows. In Section \ref{sec:definitions}, we present the basic definitions used throughout
the paper, and we define the problems solved. In Section \ref{sec:properties}, we prove some properties of seeds, covers, periods and borders
used later for finding quasiperiodicities in Fibonacci strings. In Section \ref{sec: quasiperiodicities} we identify quasiperiodicities in Fibonacci strings and circular Fibonacci strings. Finally, we give some future proposals and a brief conclusion in Section \ref{sec:conclusion}.

\section{Definitions and Problems}\label{sec:definitions}

Throughout this paper we consider a string $y$ of length $|y| = n$, $n>0$, on a
fixed alphabet. It is represented as $y[1\dd n]$.
A string $w$ is a {\em factor} of $y$ if $y=uwv$ for two strings $u$ and $v$.
It is a {\em prefix} of $y$ if $u$ is empty and a {\em suffix} of $y$
if $v$ is empty.
A string $u$ is a \textit{border} of $y$ if $u$ is both a prefix and a suffix of $y$.
{\em The border} of $y$, denoted by $\textit{border}(y)$, is the length of the longest border of $y$.
A string $u$ is a \textit{period} of $y$ if $y$ is a prefix of $u^k$ for
some positive integer $k$, or equivalently if $y$ is a prefix of $uy$.
{\em The period} of $y$, denoted by $\textit{period}(y)$, is the length of the shortest period of $y$.
For a string $u=u[1\dd m]$ such that $u$ and $v$ share a common part
 $u[m-\ell+1 \dd m]=v[1\dd \ell]$ for some $1 \leq \ell \leq m$, the string
 $u[1\dd m]v[\ell+1\dd n]=u[1\dd m-\ell]v[1 \dd n]$ is called
 a {\em superposition} of $u$ and $v$ with an \textit{overlap} of length $\ell$.
A string $x$ of length $m$ is a {\em cover} of $y$ if both
$m<n$ and there exists a set of positions $\cpos \subseteq \{1, \ldots, n-m+1\}$
that satisfies both $y[i\dd i+m-1]=x$ for all $i\in \cpos$ and
$\bigcup_{i\in \cpos} \{ i,\ldots,i+m-1 \} = \{1,\ldots,n\}$.
A string $v$ is a \emph{seed} of $y$, if it is a cover
of a superstring of $y$, where a superstring of $y$ is a string of form $uyv$ and $u,v$ are possiblu empty strings. A left seed of a string $y$ is a prefix of $y$ that is
a cover of a superstring of $y$ of the form $yv$, where $v$ is a possibly empty string.
Similarly a right seed of a string $y$ is a suffix of $y$ that is
a cover of a superstring of $y$ of the form $vy$, where $v$ is a possibly empty string.\\
\noindent
We define a (finite) Fibonacci string $F_{k}$ , $k\in\{0,1,2,\dots\}$ , as follows: \\
$F_{0} = b,\quad F_{1} = a,\quad F_{n}=F_{n-1}F_{n-2} \quad n\in\{2,3,4,\dots\}$\\
A (finite) circular Fibonacci string $C(F_{k})$ , $k\in\{0,1,2,\dots\}$ , is made by concatenating the first letter of $F_{k}$ to its last letter.
As before a substring $u$ of $C(F_{k})$ is a cover of $C(F_{k})$ if every letter of $C(F_{k})$ falls within an occurrence of $u$ within $C(F_{k})$.\\

\noindent
We consider the following problems:
\begin{problem}Identify all left seeds of some Fibonacci string $F_{n}$.
\end{problem}
\begin{problem}Identify all right seeds of some Fibonacci string $F_{n}$.
\end{problem}
\begin{problem}Identify all covers of some Fibonacci string $F_{n}$.
\end{problem}
\begin{problem}Identify all seeds of some Fibonacci string $F_{n}$.
\end{problem}
\begin{problem}Identify all covers of some circular Fibonacci string $C(F_{n})$.
\end{problem}


\section{Properties}\label{sec:properties}

In this section, we prove and also quote some properties for the covers, the left/right seeds and the seeds of a given string as well as some facts on Fibonacci strings that will prove useful later on the solution of the problems that we are considering.

\begin{lemma}\cite{christou2011efficient}
A string $z$ is a left seed of $y$ iff it is a cover of a prefix of $y$ whose length is at least the period of $y$.
\label{lemma ls condition}
\end{lemma}
\begin{proof}
Direct: Suppose a string $z$ is a cover of a prefix of $y$, say $uv$, larger or equal to $\text{period}(y)$, where $|u|=\text{period}(y)$ and $v$ is
a possibly non empty string. Let $k$ the smallest integer such that $y$ a prefix of $u^{k}$. Then $z$ is a cover of $u^{k}v=ywv$, for some string $w$,
possibly empty. Therefore $z$ is a left seed of $y$.\\
\noindent Reverse: Let $z$ be a left seed of $y$.
\begin{itemize}
 \item if $|z|\leq \textit{border}(y)$. Then a suffix $v$ of $z$ (possibly empty) is a prefix of the border (consider the left seed that covers $y[\textit{period}(y)]$). Then $z$ is a cover of $uv$, where $u$ is the period of $y$.
 \item if $|z| > \textit{border}(y)$. Let $z$ not a cover of a prefix of $y$ larger or equal to $|\textit{period}(y)|$.
Let $v$ a border of $y$ such that $|v| = \textit{border}(y)$. Then $v$ is a factor of $z$, such that $z=uvw$, where $u$ and $w$
are non empty words (consider the left seed that covers $y[\textit{period}(y)]$). This gives $uv$ a longest border for $y$, which is a contradiction.\qed
\end{itemize}
\end{proof}

\begin{lemma}\cite{christou2011right}
A string $z$ is a right seed of $y$ iff it is a cover of a suffix of $y$ whose length is at least the period of $y$.
\label{lemma rs condition}
\end{lemma}
\begin{proof}
Similar to the proof of Lemma \ref{lemma ls condition}. \qed
\end{proof}

\begin{lemma}\cite{moore1994optimal}
Let $u$ be a proper cover of $x$  and let $z\neq u$ be a substring of $x$ such that $|z|\leq |u|$. Then $z$ is a cover of $x$ if and only if $z$ is a cover of $u$.
\label{lemma cover of a cover}
\end{lemma}
\begin{proof}
Clearly if $z$ is a cover of $u$ and $u$ is a cover of $x$ the $z$ is a cover of $x$. Suppose now that both $z$ and $u$ are covers of $x$. Then $z$ is a border of $x$ and hence of $u$ ($|z|\leq |u|$); thus $z$ must also be a cover of $u$.\qed
\end{proof}

\begin{lemma}\label{lem all borders Fib}\cite{cummings1996borders}
 \begin{equation}\text{All borders of $F_{n}$ are:}
            \begin{cases}
             \{\}, & \text{$n=\{0,1,2\}$}\\
             \{F_{n-2},F_{n-4},F_{n-6},\dots F_{1}\}, & \text{$n=2k+1$,$k\geq1$}\\
	     \{F_{n-2},F_{n-4},F_{n-6},\dots F_{2}\}, & \text{$n=2k$,$k\geq2$}
            \end{cases}
  \end{equation}
\end{lemma}

\begin{lemma}\cite{iliopoulos1997characterization}\label{lem expansion Fib}
For any integer $k\geq2$, $F_{k}=P_{k}\delta_{k}$, where $P_{k}=F_{k-2}F_{k-3}\dd F_{1}$ and $\delta_{k}=ab$ if $k$ is even, and $\delta_{k}=ba$ otherwise.
\end{lemma}
\begin{proof}
Easily proved by induction. \qed 
\end{proof}

It is sometimes useful to consider the expansion of a Fibonacci string as a concatenation of two Fibonacci substrings. We define the $F_{m},F_{m-1}$ expansion of $F_{n}$, where $n \in \{2,3,\dots \}$ and $m\in \{1,2,\dots,n-1\}$, as follows:
\begin{itemize}
 \item Expand  $F_{n}$ using the recurrence formula as $F_{n-1}F_{n-2}$.
 \item Expand  $F_{n-1}$ using the recurrence formula as $F_{n-2}F_{n-3}$.
 \item Keep expanding as above until $F_{m+1}$ is expanded.
\end{itemize}

\begin{lemma}
The $F_{m},F_{m-1}$ expansion of $F_{n}$, where $n \in \{2,3,\dots \}$ and $m \in \{1,2,\dots,n-1 \}$ is unique.
\end{lemma}
\begin{proof}
Easily proved by induction. \qed
\end{proof}

\begin{lemma}\label{lem Fm in Fn}
The starting positions of the occurrences of $F_{m}$ in $F_{n}$ are the starting positions of the factors considered in the $F_{m},F_{m-1}$ expansion of $F_{n}$, where $n \in\{2,3,\dots\}$ and $m\in\{1,2,\dots,n-2\}$ except of the last $F_{m-1}$, if it is a border of $F_{n}$.
\end{lemma}
\begin{proof}
Using the recurrence relation we can get the $F_{m},F_{m-1}$ expansion of $F_{n}$ as shown before:\\
$F_{n}=F_{m}F_{m-1}F_{m}F_{m}F_{m-1}F_{m}F_{m-1}F_{m}\dots$\\
We can now observe many occurrences of $F_{m}$ in $F_{n}$. Any other occurrence should have one of the following forms (note that there are no consecutive $F_{m-1}$ in the above expansion):
\begin{itemize}
\item $xy$, where $x$ is a non empty suffix of $F_{m}$ and $y$ a non empty prefix of $F_{m-1}$. Then both $x$ and $y$ are borders of $F_{m}$. It holds that $|x|+|y|=|F_{m}|=|F_{m-1}|+|F_{m-2}|$, but $|x|\leq |F_{m-2}|$, $|y|\leq |F_{m-2}|$ and so there exist no such occurrence of $F_{m}$ in $F_{n}$.
\item $xy$, where $x$ is a non empty suffix of $F_{m-1}$ and $y$ a non empty prefix of $F_{m}$. Then $y$ is also a border of $F_{m}$ and so belongs to $\{F_{m-2},F_{m-4},\dots F_{3}\}$, if $n$ is odd, or to $\{F_{m-2},F_{m-4}\dots F_{4}\}$, otherwise. But $|x|+|y|=|F_{m}|$ and $0<|x|\leq|F_{m-1}|$ so in either case the only solution is $x=F_{m-1}$ and $y=F_{m-2}$ giving the occurrences of $F_{m}$ at the starting positions of $F_{m-1}$ in the above expansion. 
\item $xF_{m-1}y$, where $x$ is a non empty suffix of $F_{m}$ and $y$ a non empty prefix of $F_{m}$. Then both $x$ and $y$ are borders of $F_{m}$. It holds that $|x|+|y|=|F_{m-2}|$, but as both $x$ and $y$ are non empty $|x|\leq |F_{m-4}|$, $|y|\leq |F_{m-4}|$ and so there exist no such occurrence of $F_{m}$ in $F_{n}$.
\item $xy$, where $x$ is a non empty suffix of $F_{m}$ and $y$ a non empty prefix of $F_{m}$(note that there is no such occurrence in the $F_{n-2},F_{n-1}$ expansion). Then both $x$ and $y$ are borders of $F_{m}$. It holds that $|x|+|y|=|F_{m}|$, but as both $x$ and $y$ are non empty $|x|\leq |F_{m-2}|$, $|y|\leq |F_{m-2}|$ and so there exist no such occurrence of $F_{m}$ in $F_{n}$. \qed
\end{itemize}

\end{proof}

\begin{lemma}\label{lem invalid ls}
 For every integer $n\geq5$, $F_{n}[1\dd|F_{n-1}|-1]$ is not a left seed of $F_{n}$. 
\end{lemma}
\begin{proof}
Using the recurrence relation we can expand $F_{n}$ , $n\geq5$ , in the following two ways:\\
$F_{n}=F_{n-2}F_{n-3}F_{n-2}=F_{n-2}F_{n-2}F_{n-5}F_{n-4}$\\
Then one can see that $x=F_{n}[1\dd|F_{n-1}-1|]=F_{n-2}P_{n-3}\delta_{n-3}[1]$(Lemma \ref{lem expansion Fib}). 
Using Lemma \ref{lem Fm in Fn} we can see that by expanding $x$ from the prefix and suffix positions of $F_{n-2}$ we cover $F_{n}$ except $F_{n}[|F_{n-1}|-1]$. Expanding $F_{n-2}$ from its middle occurrence yields the factor $y=F_{n-2}F_{n-5}P_{n-4}\delta_{n-4}[1]=F_{n-2}P_{n-3}\delta_{n-4}[1]$. It is easy to see that $x$ an $y$ differ at their last letter and hence the above result follows. \qed

\end{proof}

\begin{lemma}\label{lem invalid rs}
 For every integer $n\geq5$, $xF_{n-4}$,  where $x$ is a suffix of $F_{n-3}$ and $0<|x|<F_{n-3}$, is not a right seed of $F_{n}$. 
\end{lemma}
\begin{proof}
Using the recurrence relation we can expand $F_{n}$ , $n\geq5$ , in the following way:\\
$F_{n}=F_{n-4}F_{n-5}F_{n-4}F_{n-4}F_{n-5}F_{n-4}F_{n-5}F_{n-4}$\\
Then any right seed of form $xF_{n-4}$, $0<|x|<|F_{n-3}|$, has $x$ as a suffix of $F_{n-4}F{n-5}$. Clearly the $3$ occurrences of $F_{n-4}$ at the starting positions of $F_{n-5}$ (Lemma \ref{lem Fm in Fn}) cannot be expanded to their left to give right seeds as an $F_{n-4}$ is to their left, which has a different ending than that of $F_{n-5}$ (Lemma \ref{lem expansion Fib}). Then $F_{n}[|F_{n-4}|+|F_{n-5}|+|F_{n-4}|+1]$ can not be covered by expanding the other $5$ occurrences of $F_{n-4}$ in $F_{n}$.\qed
\end{proof}


\section{Quasiperiodicities in Fibonacci strings}
\label{sec: quasiperiodicities}
In this section we identify quasiperiodicities on Fibonacci strings (left seeds, right seeds, seeds, covers) and circular Fibonacci strings (covers).\\

Identifying all covers of a Fibonacci string is made easy by identifying the longest cover of the string and then applying Lemma \ref{lemma cover of a cover} as shown in the theorem below.

\begin{theorem}\label{thm all covers Fib}
 \begin{equation}\text{All covers of $F_{n}$ are:}
            \begin{cases}
             F_{n}, & \text{$n=\{0,1,2,3,4\}$}\\
             \{F_{n}, F_{n-2},F_{n-4},F_{n-6},\dots F_{3}\}, & \text{$n=2k+1$,$k\geq2$}\\
	     \{F_{n}, F_{n-2},F_{n-4},F_{n-6},\dots F_{4}\}, & \text{$n=2k$,$k\geq3$}
            \end{cases}
  \end{equation}

\end{theorem}
\begin{proof}
It is easy to see that the theorem holds for $n=\{0,1,2,3,4\}$.
Using the recurrence relation we can expand $F_{n}$ , $n\geq5$ , in the following two ways:\\
$F_{n}=F_{n-2}F_{n-3}F_{n-2}=F_{n-2}F_{n-2}F_{n-5}F_{n-4}$\\
It is now obvious that $F_{n-2}$ is a cover of $F_{n}$. By Lemma \ref{lem all borders Fib} $F_{n-2}$ is also the longest border of $F_{n}$ and therefore the second longest cover of $F_{n}$ (after $F_{n}$). Similarly $F_{n-4}$ is the longest cover of $F_{n-2}$, $F_{n-6}$ is the longest cover of $F_{n-4}$, etc. Hence by following Lemma \ref{lemma cover of a cover} we get the above result. \qed
\end{proof}

Identifying left seeds of a Fibonacci string $F_{n}$ is made possible for large $n$ by characterizing each possible left seed as a substring of the form $F_{m}x$, where $m\in\{3,\dots n-1\}$ and $x$ a possibly empty prefix of $F_{m-1}$. We then use the $F_{m},F_{m-1}$ expansion of $F_{n}$ along with Lemma \ref{lem invalid ls} and the following result follows.

\begin{theorem}\label{thm all ls Fib}
All left seeds of $F_{n}$ are:
\begin{itemize}
 \item $F_{n}$, if $n=\{0,1,2\}$
 \item $\{ab,aba\}$, if $n=3$
 \item $\{F_{n-1}x$: $x$ a possibly empty prefix of $F_{n-2}\}$ $\bigcup^{n-2}_{m=3}\{F_{m}x$: $x$ a possibly empty prefix of $F_{m-1}[1 \dd |F_{m-1}|-2]\}$, if $n\geq4$
\end{itemize}

\end{theorem}
\begin{proof}
It is easy to see that the theorem holds for $n=\{0,1,2,3,4\}$.
For even $n\geq 5$ by Theorem \ref{thm all covers Fib} $\{F_{n},F_{n-2},F_{n-4},\dots F_{4}\}$ are covers of $F_{n}$ and therefore left seeds of $F_{n}$. Again by Theorem \ref{thm all covers Fib} $\{F_{n-1},F_{n-3},F_{n-5},\dots F_{3}\}$ are all covers of $F_{n-1}$ which is the period of $F_{n}$  and hence by Lemma \ref{lemma ls condition} $\{F_{n},F_{n-1},F_{n-2},\dots F_{3}\}$ are left seeds of $F_{n}$. By making similar observations for odd $n\geq 5$ we get that for $\{F_{n},F_{n-1},F_{n-2},\dots F_{3}\}$ are all left seeds of $F_{n}$ in either case. Only $a$ and $ab$ might be shorter left seeds but they are rejected as they are not left seeds of $F_{4}$ and so they are not left seeds of any longer Fibonacci string ($F_{4}$ is a prefix of every other $F_{n}$, $n\geq5$). Therefore the remaining left seeds are of the form $F_{m}x$, where $m\in\{3,4,\dots,n-1\}$ and $0<|x|<|F_{m-1}|$.
Using the recurrence relation we can get the $F_{m},F_{m-1}$ expansion of $F_{n}$ ,for any $m\in\{3,4,\dots,n-1\}$, as shown before:\\
$F_{n}=F_{m}F_{m-1}F_{m}F_{m}F_{m-1}F_{m}F_{m-1}F_{m}\dots$\\
We then try to expand the seed from each $F_{m},F_{m-1}$ in the above expansion as of Lemma \ref{lem Fm in Fn} (note that there are no consecutive $F_{m-1}$ in the above expansion).\\
$F_{m}F_{m-1}=F_{m}P_{m-1}\delta_{m-1}$\\
$F_{m}F_{m}=F_{m}F_{m-1}F_{m-2}=F_{m}P_{m-1}\delta_{m-1}F_{m-2}$\\
$F_{m-1}F_{m}=F_{m-1}F_{m-2}F_{m-3}F_{m-2}=F_{m}F_{m-3}P_{m-2}\delta_{m-2}=F_{m}P_{m-1}\delta_{m-2}$\\
It is now obvious that any $F_{m}x$, where $m\in\{3,4,\dots,n-1\}$ and $0<|x|\leq|F_{m-1}|-2$ is a left seed of $F{n}$. $F_{n-1}F_{n-2}[1\dd|F_{n-2}|-1]$ is the only other left seed as it covers the period of $F_{n}$ (Lemma \ref{lemma ls condition}). 
That there are no left seeds of form $F_{m}x$, where $m\in\{3,4,\dots,n-2\}$ and $|x|=|F_{m-1}|-1$ follows from Lemma \ref{lem invalid ls}. \qed

\end{proof}

Identifying right seeds of a Fibonacci string $F_{n}$ is made possible for large $n$ by characterizing each possible right seed as a substring of the form $xF_{m}$, where $m\in\{3,5,\dots n-2\}$ if $n$ is odd or $m\in\{4,6,\dots n-2\}$ if $n$ is even, and $x$ is a possibly empty suffix of $F_{m+1}$. We then use the $F_{m},F_{m-1}$ expansion of $F_{n}$ along with Lemma \ref{lem invalid rs} and the following result follows.

\begin{theorem}\label{thm all rs Fib}
All right seeds of $F_{n}$ are:
\begin{itemize}
 \item $F_{n}$, if $n=\{0,1,2\}$
 \item $\{F_{n},F_{n-2},F_{n-4},F_{n-6},\dots F_{3}\}$ $\cup$ $\{xF_{n-3}F_{n-2}$: $x$ a possibly empty suffix of $F_{n-2}\}$, if $n=2k+1$, $k\geq1$
 \item $\{F_{n},F_{n-2},F_{n-4},F_{n-6},\dots F_{4}\}$ $\cup$ $\{xF_{n-3}F_{n-2}$: $x$ a possibly empty suffix of $F_{n-2}\}$, if $n=2k$, $k\geq2$
\end{itemize}

\end{theorem}
\begin{proof}
It is easy to see that the theorem holds for $n=\{0,1,2,3,4\}$.
For even $n\geq 5$, by Theorem \ref{thm all covers Fib} $\{F_{n},F_{n-2},F_{n-4},\dots F_{4}\}$ are covers of $F_{n}$ and therefore right seeds of $F_{n}$. Only $\{baab,aab,ab,b\}$ might be shorter right seeds but they are rejected as they are not right seeds of $F_{6}$ and so they are not right seeds of any $F_{n}$, where $n$ even and $n\geq5$ ($F_{6}$ is a suffix of every other $F_{n}$, $n$ even and $n\geq5$). 
Similarly for odd $n\geq5$ $\{F_{n},F_{n-2},F_{n-4},\dots F_{3}\}$ are right seeds of $F_{n}$ and $F_{3}$ is its shortest right seed.
Therefore the remaining right seeds are of the form $xF_{m}$, where $0<|x|<|F_{m+1}|$ and $m\in\{4,6,\dots,n-2\}$, if $n$ is even, or $m\in\{3,5,\dots,n-2\}$, otherwise.\\
The only other right seeds  are of the form $xF_{n-3}F_{n-2}$ ,where $x$ is a suffix of $F_{n-2}$ and $0\leq|x|<F_{n}$, as it is easy to see that they cover the period of $F_{n}$ (Lemma \ref{lemma rs condition}). \\
The fact that there are no right seeds of form $xF_{n-2}$, where $0<|x|<|F_{n-3}|$, follows from Lemma \ref{lem Fm in Fn}. Clearly the middle occurrence of $F_{n-2}$ cannot be expanded to the left as an $F_{n-2}$ is to its left, which has a different ending than that of $F_{n-3}$ at the left of the last $F_{n-2}$. Then $F_{n}[|F_{n-2}|+1]$ can not be covered by the expanding the other $2$ occurrences of $F_{n-2}$ in $F_{n}$.\\
The fact that there are no right seeds of of the form $xF_{m}$, where $0<|x|<|F_{m+1}|$ and $m\in\{4,6,\dots,n-4\}$, $n$ is even, or $m\in\{3,5,\dots,n-4\}$, otherwise, follows from Lemma \ref{lem invalid rs}.  \qed
\end{proof}

Identifying all seeds of a Fibonacci string $F_{n}$ is made possible for large $n$ by characterizing each possible seed as a substring of the form $xF_{m}y$, where $m\in\{3,4,\dots n-1\}$ and $x$, $y$ follow some restrictions such that $F_{m}$ is the longest Fibonacci substring in the seed and no occurrence of $F_{m}$ in the seed starts from a position in $x$. We then use the $F_{m},F_{m-1}$ expansion of $F_{n}$ along with Lemma \ref{lem Fm in Fn} and the result below follows.

\begin{theorem}\label{thm all seeds Fib}
All seeds of $F_{n}$ are:
\begin{itemize}
 \item all left/right seeds of $F_{n}$, if $n=\{0,1,2,3\}$
 \item all left/right seeds of $F_{n}$ and $baa$, if $n=4$
 \item all left/right seeds of $F_{n}$,\\
 strings of form $\{xF_{m}y$: $x$ a suffix of $F_{m}$,$y$ a prefix of $F_{m-1}$,$0<|x|<|F_{m}|$,$0<|y|<|F_{m-1}|-1$,$|x|+|y|\geq F_{m-1}$ and $m\in\{3,\dots,n-3\}\}$,\\
 strings of form $\{xF_{m-1}F{m}y$: $x$ a suffix of $F_{m}$,$y$ a prefix of $F_{m-1}$,$|x|+|y|\geq F_{m}$ and $m\in\{3,\dots,n-3\}\}$,\\
 strings of form $\{xF_{n-2}y$: $x$ a suffix of $F_{n-2}$,$y$ a prefix of $F_{n-5}F_{n-4}$,$0<|x|<|F_{n-2}|$,$0 < |y| \leq |F_{n-3}|$ and $|x|+|y| \geq |F_{n-3}|\}$, if $n\geq5$
\end{itemize}

\end{theorem}
\begin{proof}
It is easy to see that the theorem holds for $n=\{0,1,2,3,4\}$.
For $n\geq 5$ it is obvious that all left seeds of $F_{n}$ and all right seeds of $F_{n}$ are also seeds of $F_{n}$.\\
Therefore the remaining seeds are of the form $xF_{m}y$, such that $F_{m}$ is the leftmost occurrence of the longest Fibonacci word present in the seed, $m\in\{3,4,\dots,n-2\}$, $|x|>0$ and $|y|>0$.\\
For $m=n-2$ the expansion of $F_{n}=F_{n-2}F_{n-3}F_{n-2}=F_{n-2}F_{n-2}F_{n-5}F_{n-4}$ is very small so we consider it separately. 
By expanding the middle occurrence of $F_{n-2}$ we get the seed $xF_{n-2}y$, where $0<|x|<|F_{n-2}|$, $0<|y|\leq|F_{n-3}|$ and $|x|+|y|\geq |F_{n-3}|$.
As of lemma \ref{lem Fm in Fn} the remaining seeds of form $xF_{m}y$, such that $F_{m}$ is the leftmost occurrence of the longest Fibonacci word present in the seed, $m\in\{3,4,\dots,n-3\}$, $|x|>0$ and $|y|>0$, have their leftmost $F_{m}$ factor occurring in the start position of either an $F_{m}$ or an $F_{m-1}$ in the $F_{m},F_{m-1}$ expansion of $F_{n}=F_{m}F_{m-1}F_{m}F_{m}F_{m-1}F_{m}F_{m-1}F_{m}\dots$. We consider the following cases (note that there are no consecutive $F_{m-1}$ in the above expansion):
\begin{itemize}
\item A seed of form $xF_{m}y$, such that $F_{m}$ has a $F_{m-1}$ to its left in the $F_{m},F_{m-1}$ expansion of $F_{n}$ and $0<|x|<F_{m-1}$ (otherwise there exist a new leftmost occurrence of $F_{m}$ in the seed). The occurrences of $F_{m}$ that we are considering have starting positions only from a $F_{m}$ in the expansion of $F_{n}$, then $y$ can be up to $F_{m-1}[1\dd |F_{m-1}|-1]$ (otherwise a $F_{m+1}$ is created). But such a seed fails to cover $F_{n}[|F_{m}F_{m-1}F_{m}F_{m-1}|-1]$. 
\item A seed of form $xF_{m}y$, such that $F_{m}$ has a $F_{m}$ to its left in the $F_{m},F_{m-1}$ expansion of $F_{n}$ and $0<|x|<F_{m}$ (otherwise a $F_{m+1}$ is created). If the occurrences of $F_{m}$ that we are considering have starting positions both from a $F_{m}$ and a $F_{m-1}$ in the expansion of $F_{n}$, then $y$ can be up to $F_{m-1}[1\dd |F_{m-1}|-2]$ (otherwise the factors differ). Furthermore $|x|+|y|\geq |F_{m-1}|$ ,such as to cover $F_{n}[|F_{m}F_{m-1}|+1\dd|F_{m}F_{m-1}F_{m}|]$. Such a seed covers $F_{n}$ as $F_{m+2}=F_{m}F_{m-1}F_{m}=F_{m}F_{m}P_{m-1}\delta_{m-2}$ is a left seed of $F_{n}$ (Theorem \ref{thm all ls Fib}) composing $F_{n}$ with concatenations of overlap $0$ (factors are joined by considering the seed that its leftmost $F_{m}$ starts from the next $F_{m+2}$) or $F_{m}$ (factors are joined as $|x|+|y|\geq F_{m-1}$). 
If the occurrences of $F_{m}$ that we are considering have starting positions only from a $F_{m}$ in the expansion of $F_{n}$, then $y$ can be up to $F_{m-1}[1\dd |F_{m-1}|-1]$ (otherwise a $F_{m+1}$ is created). But such a seed fails to cover $F_{n}[|F_{m}F_{m-1}|]$. 
If the occurrences of $F_{m}$ that we are considering have starting positions only from a $F_{m-1}$ in the expansion of $F_{n}$, then $|y|$ can be up to $2|F_{m-1}|-1]$ (otherwise a $F_{m+1}$ is created).  Furthermore $|x|+|y|\geq |F_{m}|+|F_{m-1}|=|F_{m+1}|$ ,such as to cover $F_{n}[|F_{m}F_{m-1}|+1\dd|F_{m}F_{m-1}F_{m}F_{m}|]$. Such a seed covers $F_{n}$ as $F_{m+2}=F_{m}F_{m-1}F_{m}=F_{m}F_{m}P_{m-1}\delta_{m-2}$ is a left seed of $F_{n}$ (Theorem \ref{thm all ls Fib}) composing $F_{n}$ with concatenations of overlap $0$ (factors are joined as $|x|+|y|\geq |F_{m+1}|$) or $F_{m}$ (factors are joined as $|x|+|y|\geq |F_{m+1}|>|F_{m-1}|$). \qed
\end{itemize}

\end{proof}

Finding all covers of a circular Fibonacci string is now obvious, we just need to check the seeds of the relevant Fibonacci string. Those which are covers of a superstring of form  $xF_{n}y$, where $x$ is a possibly empty suffix of $F_{n}$ and $y$ is a possibly empty prefix of $F_{n}$ are covers of $C(F_{n})$.

\begin{theorem}\label{thm all covers CF}
All covers of $C(F_{n})$ are:
\begin{itemize}
 \item $F_{n}$, if $n=\{0,1,2,3\}$
 \item $F_{n}$ and $F_{n-1}$, if $n=4$
 \item $F_{n}$, strings of form $\{F_{m}x$: $x$ a possibly empty prefix of $F_{m-1}[1 \dd |F_{m-1}|-2]$ and $m\in\{3,\dots,n-1\}\}$,\\
 strings of form $\{xF_{m}y$: $x$ a suffix of $F_{m}$,$y$ a prefix of $F_{m-1}$,$0<|x|<|F_{m}|$,$0<|y|<|F_{m-1}|-1$,$|x|+|y|\geq F_{m-1}$ and $m\in\{3,\dots,n-2\}\}$,\\
 strings of form $\{xF_{m-1}F{m}y$: $x$ a suffix of $F_{m}$,$y$ a prefix of $F_{m-1}$,$|x|+|y|\geq F_{m}$ and $m\in\{3,\dots,n-3\}\}$, if $n\geq5$

\end{itemize}

\end{theorem}
\begin{proof}
It is easy to see that the theorem holds for $n=\{0,1,2,3,4\}$.
For larger $n$ the covers of $C(F_{n})$ are at most the seeds of $F_{n}$. A seed is a cover of $C(F_{n})$ iff it covers a superstring of $F_{n}$ of form $xF_{n}y$, where $x$ is a possibly empty suffix of $F_{n}$ and $y$ is a possibly empty prefix of $F_{n}$. We consider the following cases:
\begin{itemize}
\item Left seeds of form $F_{n}[1\dd |F_{k}|+i]$, where $i\in\{0,1,\dots,|F_{k-1}|-2\}$ and $k\in \{3,4,\dots n-1 \}$, are covers of $F_{n}F_{k}[1\dd i]$ ,if $F_{k}$ is a cover of $F_{n}$, or covers of $F_{n}F_{k}[1\dd |F_{k-2}|+i]$ otherwise, and hence covers of $C(F_{n})$ in both cases. Clearly $F_{n}$ is also a cover of $C(F_{n})$. $F_{n}[1\dd|F_{n}|-1]$ is not a cover of $C(F_{n})$ as it fails to cover a prefix of $F_{n}F_{n}$ longer than $|F_{n}|-1$ (consider the $F_{n-1},F_{n-2}$ expansion of $F_{n}$ along with Lemma \ref{lem Fm in Fn}).
\item The only right seeds of $F_{n}$ that are covers of $C(F_{n})$ are the covers of $F_{n}$ (included above). Right seeds of form $xF_{n-3}F{n-2}$, where $x$ is a possibly empty suffix of $F_{n-2}$ and $0\leq |x| <|F_{n-2}|$, fail to cover a suffix of $F_{n}F_{n}=F{n-2}F_{n-3}F{n-2}$ longer than $|xF_{n-3}F_{n-2}|$ (consider the $F_{n-2},F_{n-3}$ expansion of $F_{n}$ along with Lemma \ref{lem Fm in Fn}), and so they are not covers of $C(F_{n})$.
\item Seeds of form $xF_{m}y$ where $x$ a suffix of $F_{m}$ and $y$ a prefix of $F_{m-1}$, $0<|x|<|F_{m}|$, $0<|y|<|F_{m-1}|-1$, $|x|+|y|\geq F_{m}$ and $m\in\{3,4,\dots,n-3\}$ are covers of $xF_{n}F_{m}y$ ,if $F_{m}$ is a cover of $F_{n}$, or covers of $xF_{m-1}F_{n}F_{m-2}y$ otherwise, and hence covers of $C(F_{n})$ in both cases.
\item Seeds of form $xF_{m-1}F{m}y$ where $x$ a suffix of $F_{m}$ and $y$ a prefix of $F_{m-1}$, $0<|x|<|F_{m}|$, $0<|y|<|F_{m-1}|$, $|x|+|y|\geq F_{m}$ and $m\in\{3,4,\dots,n-3\}$ are covers of $xF_{m-1}F_{m}F_{n}y$ ,if $F_{m}$ is a cover of $F_{n}$, or covers of $xF_{m-1}F_{n}F_{m}y$ otherwise, and hence covers of $C(F_{n})$ in both cases. 
\item Seeds of form $xF_{n-2}y$ where $x$ a suffix of $F_{n-2}$ and $y$ a prefix of $F_{n-5}F_{n-4}$ such that $0<|x|<|F_{n-2}|$, $0<|y|<|F_{n-3}|-1$ and $|x|+|y|\geq |F_{n-3}|$ are covers of $xF_{n}F_{n-2}y$ and hence covers of $C(F_{n})$. When $y=F_{n-5}F_{n-4}$ or $F_{n-5}F_{n-4}[1\dd |F_{n-4}|-1]$ the seed fails to cover $xF_{n}F_{n-2}y$, the first $F_{n-2}$ of $F_{n}$ can not be expanded further to the right. Trying to force an overlap of $xF_{n-2}y$ to the left of $F_{n}F_{n-2}y$ gives the superstrings $xF_{n-3}F_{n}F_{n-2}y$ and $xF_{n-2}F_{n-5}F_{n}F_{n-2}y$ (consider the occurrences of $F_{n-4}$ in $F_{n}$), which are not made of suffixes of $F_{n}$, as clearly $F_{n-3}$ and $F{n-5}$ are not borders of $F_{n}$ (Lemma \ref{lem all borders Fib}), and so they are not covers of $C(F_{n})$.
\end{itemize}
\qed

\end{proof}

\section{Conclusion and Future Work}\label{sec:conclusion}

In this paper, we have identified all left seeds, right seeds, seeds and covers of every Fibonacci string as well as all covers of a circular Fibonacci string under the restriction that these quasiperiodicities are also substrings of the given Fibonacci string.
Beyond their obvious theoretical interest, those results might prove useful in testing algorithms that find quasiperiodicities in strings and giving worst case examples on them or in extending the above work in general Sturmian words (the infinite Fibonacci word, a word which has every Fibonacci word as a prefix, is Sturmian).

\bibliographystyle{abbrv}
\bibliography{BIB}

\end{document}